\documentclass[11pt,dvips,twoside,letterpaper]{article}

\usepackage{pslatex}
\usepackage{fancyhdr}
\usepackage{graphicx}
\usepackage{geometry}

\def\figurename{Figure} 
\makeatletter
\renewcommand{\fnum@figure}[1]{\figurename~\thefigure.}
\makeatother

\def\tablename{Table} 
\makeatletter
\renewcommand{\fnum@table}[1]{\tablename~\thetable.}
\makeatother

\usepackage{amsmath}
\usepackage{amssymb}
\usepackage{amsfonts}
\usepackage{amsthm,amscd}

\newtheorem{thm}{Theorem}[section]
\newtheorem{prop}[thm]{Proposition}

\newtheorem{lem}[thm]{Lemma}
\newtheorem{defn}[thm]{Definition}

\theoremstyle{definition}

\newcommand{\bR}{{\mathbb R}}

\newcommand{\bZ}{{\mathbb Z}}

\def\H{{\mathcal H}}

\def\M{{\mathcal M}}

\def\bmo{{\mathfrak {bmo}}}
\def\rbmo{{\mathfrak {rbmo}}}
\def\h{{\mathfrak h}}

\def\supp{{\rm supp}}

\def\X{{\mathcal X}}

\def\BMO{B\! M\! O}

\def\RBMO{R\! B\! M\! O}

\theoremstyle{notation}

\numberwithin{equation}{section}

\begin{document}
\title{\bfseries\scshape{Product of functions in $\BMO$ and $\H^{1}$ in non-homogeneous spaces}}
\author{\bfseries\scshape Justin Feuto\thanks{E-mail address: justfeuto@yahoo.fr}\\
UFR Math\'ematiques et Informatique, Universit\'e de Cocody, \\22BP1194 Abidjan, C\^ote d'Ivoire\\
}

\date{}
\maketitle \thispagestyle{empty} \setcounter{page}{1}

\thispagestyle{fancy} \fancyhead{}
\renewcommand{\headrulewidth}{0pt}

\begin{abstract}
Under the assumption that the underlying measure is a non-negative Radon
measure which only satisfies some growth condition and may not be doubling, we define the product of functions in the regular $BMO$ and the atomic block $\H^{1}$ in the sense of distribution, and show that this product may be split into two parts, one in $L^{1}$ and the other in some Hardy-Orlicz space.
\end{abstract}


\vspace{.08in} \noindent \textbf{Keywords}: local Hardy space, local BMO space, atomic block, block, non doubling measure, Hardy-Orlicz space.

\newpage

\pagestyle{fancy} \fancyhead{} \fancyhead[EC]{J. Feuto,} \fancyhead[EL,OR]{\thepage} \fancyhead[OC]{Product $\RBMO(\mu)-\H^{1}(\mu)$} \fancyfoot{}
\renewcommand\headrulewidth{0.5pt}

\section{Introduction }\label{s1}
In their paper \cite{BIJZ}, Bonami, Iwaniec, Jones and Zinsmeister defined the product of functions  $f\in\BMO(\bR^{n})$ and $h\in\H^{1}(\bR^{n})$ as a distribution operating on a test function $\varphi\in\mathcal D(\bR^{n})$ by the rule
\begin{equation}
\left\langle f\times h,\varphi\right\rangle:=\left\langle f\varphi,h\right\rangle.\label{product}
\end{equation}
 They  proved that such distribution can be written as the sum of a function in $L^{1}(\bR^{n})$ and a distribution in a Hardy-Orlicz space $\H^{\wp}(\bR^{n},\nu)$ where 
\begin{equation}
\wp(t)=\frac{t}{\log(e+t)} \text{ and } d\nu(x)=\frac{dx}{\log(e+\left|x\right|)}.
\end{equation}
 Bonami and Feuto in \cite{BF} considered the case where $\BMO(\bR^{n})$ is replaced by its local version $\bmo(\bR^{n})$ introduced by Golberg in \cite{G}, and proved that in this case, the weighted Hardy-Orlicz space is replaced by a space of amalgam type in the sense of Wiener \cite{W}. Following the idea in \cite{BIJZ} and \cite{BF}, the author in \cite{Fe} generalized this result  in the setting of space of homogeneous type $(\X,d,\mu)$. We recall that a space of homogeneous type is a non-empty set $\X$ equipped with a quasi metric $d$ and a positive Radon measure $\mu$ such that 
\begin{equation}
\mu\left(B(x,2r)\right)\leq C\mu(B(x,r)),\ x\in\X,\ r>0
\end{equation}
 where $B(x,r)=\left\{y\in\X:d(x,y)<r\right\}$ is the ball centered at $x$ and having radius $r$. 
 
 This doubling condition is an essential assumption for most results in classical function spaces, Calder\'on-Zygmund theory and operators theory. However, it has been shown  recently (see \cite{MMNO}, \cite{T1}, \cite{T2}, \cite{Y} and \cite{HL}, and the reference therein) that one can drop the doubling condition and still obtain interesting results in the classical Calder\'on-Zygmund theory and on the classical Hardy and $\BMO$ spaces. In particular, Tolsa in \cite{T1} introduced, when the measure satisfies only the growth condition (\ref{nondoubling}), the regular bounded mean oscillation space $\RBMO(\mu)$ and its predual space $\H^{1,\infty}_{atb}(\mu)$. He showed that these spaces have similar properties to those of the classical $\BMO$ and $\H^{1}$ defined for doubling measures. 
 
 The purpose of this paper is to define the product of function in $\RBMO(\mu)$ and $\H^{1,\infty}_{atb}(\mu)$ in the sense of distribution, and to prove that some results obtained in \cite{BF}, \cite{Fe} and \cite{BIJZ} are valid in this context. To make our idea clear, let us give some notations and definitions.
 
Let $n,d$ be some fixed integers with $0 < n\leq d$. We consider $(\bR^{d},\left|\cdot\right|,\mu)$, where $\left|\cdot\right|$ is the Euclidean metric and $\mu$ a positive Radon measure that only satisfies the following growth condition  
\begin{equation}
\mu\left(B(x,r)\right)\leq C_{0} r^{n}, \text{ for all }x\in\bR^{d}\text{ and }r > 0, \label{nondoubling}
\end{equation} 
where $C_{0}>0$ is an absolute constant. Throughout the  paper, by a cube $Q\subset\bR^{d}$, we mean a closed cube with sides parallel to the axis and centered at some point $x_{Q}$ of \supp $(\mu)$, and if $\left\|\mu\right\|<\infty$, we allow $Q=\bR^{d}$ too. 

If $Q$ is a cube, we denote by $\ell(Q)$ the side length of $Q$  and for $\alpha>0$, we denote $\alpha Q$ the cube with same center as $Q$, but side length $\alpha$ times as long. We will always choose the constant $C_{0}$ in (\ref{nondoubling}) such that for all cubes $Q$, we have $\mu(Q)\leq C_{0}\ell(Q)^{n}$.

For two fixed cubes $Q\subset R$ in $\bR^{d}$, set
  \begin{equation}	S_{Q,R}=1+\sum^{N_{Q,R}}_{k=1}\frac{\mu(2^{k}Q)}{\ell^{n}(2^{k}Q)}
	\end{equation}
 where  $N_{Q,R}$ is the smallest positive integer $k$ such that $\ell(2^{k}Q)\geq \ell(R)$ (in the case $R=\mathbb R^{d}\neq Q$, we set $N_{Q,R}=\infty$).

 For a fixed $\rho>1$ and  $p\in\left(1,\infty\right]$, a function $b\in L^{1}_{loc}(\mu)$ is called a $p$-atomic block if
\begin{enumerate}
	\item[(i)] there exists some cube $R$ such that supp $b\subset R$,
	\item [(ii)] $\int_{\bR^{d}}b\; d\mu=0$,
	\item [(iii)] there are functions $a_{j}$ supported on cubes $Q_{j}\subset R$ and numbers $\lambda_{j}\in\bR$ such that 
	$	b=\sum^{\infty}_{j=1}\lambda_{j}a_{j}$
	and
	\begin{equation}
	\left\|a_{j}\right\|_{L^{p}(\mu)}\leq\left(\mu(\rho Q_{j})\right)^{\frac{1}{p}-1}\left(S_{Q_{j},R}\right)^{-1},\label{a3}
	\end{equation}
\end{enumerate}
where we used the natural convention that $\frac{1}{\infty}=0$. We put 
\begin{equation}
\left|b\right|_{\H^{1,p}_{atb}(\mu)}:=\sum_{j}\left|\lambda_{j}\right|.
\end{equation}
\begin{defn}(\cite{T1})\label{HP}
 We say that $h\in\H^{1,p}_{atb}(\mu)$ if there are $p$-atomic blocks $b_{j}$ such that
\begin{equation}
h=\sum^{\infty}_{j=1}b_{j}\text{ with } \sum^{\infty}_{j=1}\left|b_{j}\right|_{\H^{1,p}_{atb}(\mu)}<\infty,\label{atomicblockdecompo}
\end{equation}
\end{defn}
The atomic block Hardy space $\H^{1,p}_{atb}(\mu)$ is a Banach space when equipped with the norm $\left\|\cdot\right\|_{\H^{1,p}_{atb}(\mu)}$ defined by 
\begin{equation}
\left\|h\right\|_{\H^{1,p}_{atb}(\mu)}=\inf\sum^{\infty}_{j=1}\left|b_{j}\right|_{\H^{1,p}_{atb}(\mu)},\ h\in\H^{1,p}_{atb}(\mu),
\end{equation}
where the infimum is taken over all possible decomposition of $h$ into atomic blocks.
 
 As it is proved in Proposition 5.1 and in Theorem 5.5 of \cite{T1}, the definition of $\H^{1,p}_{atb}(\mu)$ does not depend on $\rho$ and we have that, for all $1<p<\infty$, the spaces $\H^{1,p}_{atb}(\mu)$ are topologically equivalent to $\H^{1,\infty}_{atb}(\mu)$. So in the sequel, we shall use the notation $\H^{1}(\mu)$ instead of $\H^{1,\infty}_{atb}(\mu)$, and take $\rho=2$.
 
When $b\in L^{1}_{loc}(\mu)$ satisfies only Condition (i) and (iii) of the definition of atomic blocks, we say that it is a $p$-block and put 
$\left|b\right|_{\h^{1}_{atb}(\mu)}=\sum_{j}\left|\lambda_{j}\right|$. Moreover, we say that $h$ belongs to the local Hardy space $\h^{1,p}_{atb}(\mu)$ (see \cite{Y}), if there are $p$-atomic blocks or $p$-blocks $b_{j}$ such that
\begin{equation}
h=\sum^{\infty}_{j=1}b_{j},
\end{equation}
where $\sum^{\infty}_{j=1}\left|b_{j}\right|_{\h^{1}_{atb}(\mu)}<\infty$, $b_{j}$ is an atomic block if \supp $b_{j}\subset R_{j}$ and $\ell(R_{j})\leq 1$, and $b_{j}$ is a block if \supp $ b_{j}\subset R_{j}$ and $\ell(R_{j})> 1$. We define the $\h^{1}_{atb}(\mu)$ norm of $h$ by
\begin{equation}
\left\|h\right\|_{\h^{1}_{atb}(\mu)}=\inf\sum^{\infty}_{j=1}\left|b_{j}\right|_{\h^{1}_{atb}(\mu)},
\end{equation}
where the infimum is taken over all possible decompositions of $h$ into atomic blocks or blocks.

The definition of local Hardy space is independent of $\rho>0$ and  for $1<p<\infty$, we have $\h^{1,p}_{atb}(\mu)=\h^{1,\infty}_{atb}(\mu)$ (see Proposition 3.4 and Theorem 3.8 of \cite{Y}). This allow us to just denote it by $\h^{1}(\mu)$ and consider also $\rho=2$.

In Theorem 5.5 of \cite{T1} and  Theorem 3.8 of \cite{Y}, it is proved that the dual space of $\H^{1}(\mu)$ and $\h^{1}(\mu)$ are respectively $\RBMO(\mu)$ and its local version $\rbmo(\mu)$ (see Section 2 for more explanations about these spaces).

Let $h=\sum_{j}b_{j}$ belonging to $\H^{1}(\mu)$, where the atomic block $b_{j}$ is supported in the cube $R_{j}$ and satisfies $b_{j}=\sum_{i}\lambda_{ij}a_{ij}$  for $a_{ij}$'s and $\lambda_{ij}$'s as in the definition of atomic blocks.  For $f\in\RBMO(\mu)$, we denote by $f_{\tilde{R}}$ the mean value of $f$ over the cube $\tilde{R}$, which is an appropriate dilation of the cube $R$ (see Section 2 for more explanation). We can see from the proof of Theorem \ref{main1} that the double series 
\begin{equation}
\sum^{\infty}_{j=1}\left(f-f_{\tilde{R}_{j}}\right)b_{j}=\sum^{\infty}_{j=1}\left(\sum^{\infty}_{i=1}\lambda_{ij}\left(f-f_{\tilde{R}_{j}}\right)a_{ij}\right)
\end{equation}
converges normally in $L^{1}(\mu)$, while 
\begin{equation}
\sum^{\infty}_{j=1}f_{\tilde{R}_{j}}b_{j}=\sum^{\infty}_{j=1}\left(\sum^{\infty}_{i=1}f_{\tilde{R}_{j}}\lambda_{ij}a_{ij}\right)
\end{equation}
 converges in the Hardy-Orlicz space $\H^{\wp}(\nu)$, where $\wp(t)=\frac{t}{\log(e+t)}$ and $d\nu(x)=\frac{d\mu(x)}{\log(e+\left|x\right|)}$. Since both convergence implies convergence in the sense of distribution, we define the product of $f$ and $h$ as the sum of both series by
\begin{equation}
f\times h=\sum^{\infty}_{j=1}\left(f-f_{\tilde{R}_{j}}\right)b_{j}+\sum^{\infty}_{j=1}f_{\tilde{R}_{j}}b_{j}.
\end{equation}
It follows that

\begin{thm}\label{main1}
For $f$  in $\RBMO (\mu)$ and $h$  in $\mathcal H^1(\mu)$, the product $f\times h$ can be given a meaning in the sense of distributions. Moreover, we have the inclusion
\begin{equation}\label{inclusion0}
  f\times h\in L^1(\mu)+ \mathcal H^\wp(\nu).
\end{equation}
\end{thm}
When we replaced $\RBMO(\mu)$ by its local version $\rbmo(\mu)$ as define in \cite{Y} (see also \cite{HYY}) we obtain the analogous of the result in \cite{BF}. We also obtain interesting results by replacing both $\RBMO(\mu)$ and $\H^{1}(\mu)$ with their local version.


The paper is organized as follows, in Section 2 we recall the definition of the space $\RBMO(\mu)$, its local version and some properties involved.

Section 3 is devoted to auxiliary results and prerequisites in Orlicz spaces while in Section 4 we give the proof of the main results and their extensions.

Throughout the paper, the letter $C$ is used for non-negative constants that may change from one occurrence to another. Constants with subscript, such as $C_{0}$, do not change in different occurrences. The notation $A\approx B$ stands for $C^{-1}A\leq B\leq CA$, $C$ being a constant not depending on the main parameters involved.
\medskip

\section{Prerequisite about $\RBMO(\mu)$, $\rbmo(\mu)$, $\H^{1}(\mu)$ and $\h^{1}(\mu)$ spaces}

\begin{defn}
Let $\alpha> 1$ and $\beta> \alpha^{n}$, we say that a cube $Q$ is an $(\alpha,\beta)$-doubling cube if $\mu\left(\alpha Q\right)\leq\beta\mu\left(Q\right)$.
 \end{defn}
 It is proved in \cite{T1} that there are a lot of "big " doubling cubes and also a lot of "small" doubling cubes, this due to the facts that $\mu$ satisfies the growth Condition (\ref {nondoubling}) and $\beta>\alpha^{n}$. More precisely, given any point $x\in$\supp $(\mu)$ and $c > 0$, there exists some $(\alpha,\beta)$-doubling cube $Q$ centered at $x$ with $\ell(Q)\geq c$.
 
On the other hand, if $\beta>\alpha^{n}$ then, for $\mu$-a.e. $x\in\bR^{d}$, there exists a sequence of $\left(\alpha,\beta\right)$-doubling cubes $\left\{Q_{k}\right\}_{k\in\mathbb N}$ centered at $x$ with $\ell\left(Q_{k}\right)\rightarrow 0$ as $k\rightarrow\infty$. 

In the following, for any $\alpha> 1$, we denote by $\beta_{\alpha}$ one of these big constants $\beta$. For definiteness, one can assume that $\beta_{\alpha}$ is twice the infimum of these $\beta$'s.

 Given $\rho>1$, we let $N$ be the smallest non-negative integer such that $2^{N}Q$ is $(\rho,\beta_{\rho})$-doubling and we denote this cube by $\tilde{Q}$.
 
 \begin{defn}(\cite{Y}) Let $\rho>1$ be some fixed constant.
 \begin{enumerate}
 \item[(a)] Let $1<\eta<\infty$.  We say that  $f\in L^{1}_{loc}(\mu)$ is in $\RBMO(\mu)$ if there exists a non-negative constant $C_{2}$ such that for any cube $Q$,
\begin{equation}\label{C1}
\frac{1}{\mu(\eta Q)}\int_{Q}\left|f(x)-f_{\tilde{Q}}\right|d\mu(x)\leq C_{2},
\end{equation}
and for any two $(\rho,\beta_{\rho})$-doubling cubes $Q\subset R$
\begin{equation}
\left|f_{Q}-f_{R}\right|\leq C_{2} S_{Q,R}.\label{C2}
\end{equation}
 Let us put
\begin{equation}
\left\|f\right\|_{RBMO(\mu)}=\inf\left\{C_{2}: (\ref{C1}) \text{ and } (\ref{C2})\text{ hold}\right\}.
\end{equation}
\item[(b)] Let $1<\eta\leq\rho<\infty$. We say that $f\in L^{1}_{loc}(\mu)$ belongs to $\rbmo(\mu)$ if there exists some constant $C_{3}$ such that (\ref{C1}) holds for any cube $Q$ with $\ell(Q)\leq 1$ and $C_{3}$ instead of $C_{2}$, (\ref{C2}) holds for any two $(\rho,\beta_{\rho})$-doubling cubes $Q\subset R$ with $\ell(Q)\leq 1$ and $C_{3}$ instead of $C_{2}$, and 
\begin{equation}
\frac{1}{\mu(\eta Q)}\int_{Q}\left|f(x)\right|d\mu(x)\leq C_{3}\label{C3}
\end{equation}
for any cube $Q$ with $\ell(Q)>1$. We set
\begin{equation}
\left\|f\right\|_{\rbmo(\mu)}=\inf\left\{C_{3}:(\ref{C1}),(\ref{C2}) \text{ and }(\ref{C3}) \text{ hold}\right\}.
\end{equation}
\end{enumerate}
\end{defn}
We should have referred to the choice of constants $\eta,\rho$ and $\beta$ in the terminology, but it is proved in \cite{T1} and \cite{Y} that $\RBMO(\mu)$ and $\rbmo(\mu)$ are independent of their choice. We also have (see Proposition 2.5 of \cite{T1} and Proposition 2.2 of \cite{Y}) that $(\RBMO(\mu),\left\|\cdot\right\|_{\RBMO(\mu)})$ and $(\rbmo(\mu),\left\|\cdot\right\|_{\rbmo(\mu)})$ are Banach spaces of functions (modulo additive constants).
 
 We have that $S_{Q,R}\approx 1+\delta\left(Q,R\right)$ (see \cite{T2}), where
 \begin{equation}
 \delta\left(Q,R\right)=\max\left(\int_{Q_{R}\setminus Q}\frac{d\mu(x)}{\left|x-x_{Q}\right|^{n}},\int_{R_{Q}\setminus R}\frac{d\mu(x)}{\left|x-x_{R}\right|^{n}}\right),
 \end{equation}
 and there exits a constant $\kappa>0$ such that for all cubes $Q\subset R$ we have
 \begin{equation}
 \delta\left(Q,R\right)\leq \kappa\left(1+\log(\frac{\ell(R)}{\ell(Q)})\right).\label{controle1}
 \end{equation}
  
 \begin{lem}
Let $f\in RBMO(\mu)$ and $\varphi\in\mathcal D(\mathbb R^{d})$. Then the pointwise product $f\varphi\in RBMO(\mu)$. Moreover, if $f\in \rbmo(\mu)$ then $f\varphi\in \rbmo(\mu)$.
\end{lem}
\begin{proof}
Let $f\in RBMO(\mu)$ and $\varphi\in\mathcal D(\mathbb R^{d})$ with support in the cube $Q_{0}$. We assume without loss of generality that $f_{\tilde{2Q_{0}}}=0$.
The point wise product $f\varphi$ belongs to $RBMO(\mu)$ if and only if for some real number $\rho>1$, there exists $C>0$ and a collection of numbers $\left\{C_{Q}(f\varphi)\right\}_{Q}$ (i.e for each cube Q, there exists $C_{Q}(f\varphi)\in\mathbb R$) such that
\begin{equation}
\int_{Q}\left|(f\varphi)(x)-C_{Q}(f\varphi)\right|d\mu(x)\leq C\label{relation1}
\end{equation}
and
\begin{equation}
\left|C_{Q}(f\varphi)-C_{R}(f\varphi)\right|\leq CS_{Q,R}\text{  for any two cubes }Q\subset R.\label{relation2}
\end{equation}

\medskip
{A-\it The choice of the numbers $C_{Q}(f\varphi)$ satisfying (\ref{relation1})}

\medskip

Let $Q$ be a cube in $\mathbb R^{d}$. If
\begin{enumerate}
\item $\mu(Q\cap Q_{0})=0$, or\label{cas1}
\item  $\mu(Q\cap Q_{0})>0$ and $Q\not\subset 2Q_{0}$\label{cas2}
\end{enumerate}
 then we take $C_{Q}(f\varphi)=0$. In the case (\ref{cas1}) we have $\int_{Q}\left|f\varphi\right|d\mu=0$ while in the case (\ref{cas2})
we have $Q_{0}\subset 5 Q$ so that 
$$\int_{Q}\left|f\varphi\right|d\mu=\int_{Q\cap Q_{0}}\left|f\varphi\right|d\mu\leq\int_{Q_{0}}\left|f\varphi\right|d\mu\leq C\left\|\varphi\right\|_{L^{\infty}}\left\|f\right\|_{RBMO(\mu)}\mu(\rho Q).$$
for any $\rho>5$.
 We suppose now that $\mu(Q\cap Q_{0})>0$ and $Q\subset 2Q_{0}$. 

We put $C_{Q}(f\varphi)=f_{\tilde{Q}}\varphi_{Q}$. It follows that
\begin{eqnarray*}
\int_{Q}\left|f\varphi-f_{\tilde{Q}}\varphi_{Q}\right|d\mu&=&\int_{Q}\left|(f-f_{\tilde{Q}})\varphi+f_{\tilde{Q}}(\varphi-\varphi_{Q})\right|d\mu\\
&\leq&\left\|\varphi\right\|_{L^{\infty}}\left\|f\right\|_{RBMO(\mu)}\mu(\rho Q)+\left|f_{\tilde{Q}}\right|\int_{Q}\left|\varphi-\varphi_{Q}\right|d\mu.
\end{eqnarray*}

 But 
 \begin{eqnarray*}
\left|f_{\tilde{Q}}\right|&=&\left|f_{\tilde{Q}}-f_{\tilde{2Q_{0}}}\right|\leq S_{Q,2Q_{0}}\left\|f\right\|_{RBMO(\mu)}\\
&\leq&C(1+\delta_{(Q,2Q_{0})})\left\|f\right\|_{RBMO(\mu)}\leq C(1+\log(\frac{2\ell(Q_{0})}{\ell(Q)}))\left\|f\right\|_{RBMO(\mu)},
\end{eqnarray*}
according to Lemma 2.4 of \cite {T2}. So that taking into consideration the following classical result
$$\int_{Q}\left|\varphi-\varphi_{Q}\right|d\mu\leq C\left\|\nabla\varphi\right\|_{L^{\infty}}\ell(Q)\mu(Q)$$
and the fact that $2\ell(Q_{0})\geq\ell(Q)$, we obtain 
\begin{eqnarray*}
\left|f_{\tilde{Q}}\right|\int_{Q}\left|(\varphi-\varphi_{Q})\right|d\mu&\leq& C(1+\log(\frac{2\ell(Q_{0})}{\ell(Q)}))\ell(Q)\mu(Q)\left\|f\right\|_{RBMO(\mu)}\\
&\leq&C\mu(Q)\left\|f\right\|_{RBMO(\mu)}.
\end{eqnarray*} 

\medskip
{B-\it Prove that the collection satisfy (\ref{relation2})}

\medskip

Let $Q\subset R$ be two cubes. If $R\cap Q_{0}=\emptyset$ or $Q\not\subset 2Q_{0}$, then $C_{Q}(f\varphi)=C_{R}(f\varphi)=0$. Thus there is nothing to prove. 

We suppose that $R\cap Q_{0}\neq\emptyset$ and $Q\subset 2Q_{0}$. 

If $R\not\subset 2Q_{0}$ then $C_{R}(f\varphi)=0$ and $Q_{0}\subset 5R$, so that
\begin{eqnarray*}
\left|f_{\tilde{Q}}\varphi_{Q}\right|&\leq&\left\|\varphi\right\|_{L^{\infty}}\left|f_{\tilde{Q}}-f_{\tilde{2Q_{0}}}\right|\leq\left\|\varphi\right\|_{L^{\infty}}S_{Q,2Q_{0}}\left\|f\right\|_{RBMO(\mu)}\\
&\leq&\left\|\varphi\right\|_{L^{\infty}}S_{Q,10R}\left\|f\right\|_{RBMO(\mu)}\leq C\left\|\varphi\right\|_{L^{\infty}}S_{Q,R}\left\|f\right\|_{RBMO(\mu)}.
\end{eqnarray*}
If $R\subset 2Q_{0}$, then 
\begin{eqnarray*}
\left|C_{R}(f\varphi)-C_{Q}(f\varphi)\right|&=&\left|f_{\tilde{R}}\varphi_{R}-f_{\tilde{Q}}\varphi_{Q}\right|\leq\left\|\varphi\right\|_{L^{\infty}}\left|f_{\tilde{R}}-f_{\tilde{Q}}\right|+\left|f_{\tilde{R}}\right|\left|\varphi_{R}-\varphi_{Q}\right|\\
&\leq&\left\|\varphi\right\|_{L^{\infty}}S_{Q,R}\left\|f\right\|_{RBMO(\mu)}+\left|f_{\tilde{R}}\right|\left|\varphi_{R}-\varphi_{Q}\right|.
\end{eqnarray*}
Let us estimate the second term.
\begin{eqnarray*}
\left|f_{\tilde{R}}\right|\left|\varphi_{R}-\varphi_{Q}\right|&\leq&C\left|f_{\tilde{R}}\right|\left(\ell(R)+\ell(Q)+\mathrm{dist}(x_{Q},x_{R})\right)\\
&\leq&C(1+\left|f_{\tilde{R}}\right|\mathrm{dist}(x_{Q},x_{R})),
\end{eqnarray*}
where $x_{Q}$ and $x_{R}$ denote the centers of the cubes $Q$ and $R$ respectively. 
But $\mathrm{dist}(x_{Q},x_{R})\leq C\ell(Q_{R})$ and $\left|f_{\tilde{R}}\right|\leq\left|f_{\tilde{Q_{R}}}\right|+\left|f_{\tilde{Q_{R}}}-f_{\tilde{R}}\right|$, which leads to 
\begin{eqnarray*}
\left|f_{\tilde{R}}\right|\mathrm{dist}(Q,R)&\leq& C\ell(Q_{R})\left(\left|f_{\tilde{Q_{R}}}\right|+\left|f_{\tilde{Q_{R}}}-f_{\tilde{R}}\right|\right)\\
&\leq&C\left\|f\right\|_{RBMO(\mu)}+S_{R,Q_{R}}\left\|f\right\|_{RBMO(\mu)}\leq C\left\|f\right\|_{RBMO(\mu)}.
\end{eqnarray*}
The result follow.
\medskip

Let us consider know the particular case where  $f$ belongs to $\rbmo(\mu)$. For any cube $Q$ such that $\ell(Q)>1$, we have
$$\left|C_{Q}(f\varphi)\right|\leq\left|f_{\tilde{Q}}\right|\left|\varphi_{Q}\right|\leq\left\|\varphi\right\|_{L^{\infty}}\left\|f\right\|_{\rbmo(\mu)}\mu(\eta \tilde{Q})/\mu(\tilde{Q})\leq C\left\|\varphi\right\|_{L^{\infty}(\mu)}\left\|f\right\|_{\rbmo(\mu)}$$
for some positive constant $C$ and fixed $1<\eta\leq\rho$, since $\ell(\tilde{Q})\geq\ell(Q)$. It follows that $f\varphi\in\rbmo(\mu)$.
\end{proof}  
  
 Inequalities of John-Nirenberg type are  valid in both spaces. More  precisely we have 
 \begin{thm}\cite{T1}
 Let $f\in\RBMO(\mu)$. For any cube $Q$ and any $\lambda>0$, we have 
 \begin{equation}
 \mu\left(\left\{x\in Q:\left|f(x)-f_{\tilde{Q}}\right|>\lambda\right\}\right)\leq C_{4}\mu(\rho Q)\exp\left(-\frac{C_{5}\lambda}{\left\|f\right\|_{\RBMO(\mu)}}\right),
 \end{equation}
 where the constants $C_{4}>0$ and $C_{5}>0$ depend only on $\rho>1$
 \end{thm}
  As we can see in Theorem 2.6 of \cite{Y}, one can replace in the previous theorem the space $\RBMO(\mu)$ by its local version $\rbmo(\mu)$ provided the cube $Q$ satisfies $\ell(Q)\leq 1$, while for cubes $Q$ such that $\ell(Q)>1$ we have $ \mu\left(\left\{x\in Q:\left|f(x)\right|>\lambda\right\}\right)\leq C_{4}\mu(\rho Q)\exp\left(-\frac{C_{5}\lambda}{\left\|f\right\|_{\rbmo(\mu)}}\right)$.  An immediate consequence of this result is that there exists a non-negative constant $C_{6}$, which can be chosen as big as we like, such that for all cube $Q$ and $\textrm{const}\equiv\!\!\!\!\!\!/\; f\in RBMO(\mu)$, 
\begin{equation}
\frac{1}{\mu\left(\rho Q\right)}\int_{Q}\exp\left(\frac{\left|f-f_{\tilde{Q}}\right|}{C_{6}\left\|f\right\|_{\RBMO(\mu)}}\right)d\mu\leq1.
\end{equation}
 We also have the following: 

\begin{lem}
Let  $\textrm{const}\equiv\!\!\!\!\!\!/\; f\in
RBMO(\mu)$ and $\mathbb Q$ the unit cube.
We have
\begin{equation}
\int_{\bR^d}\frac{\left(\exp\left(\frac{\left|f(x)-f_{\tilde{\mathbb Q}}\right|}{k}\right)-1\right)d\mu(x)}{(1+|x|)^{2n+\kappa}}\leq 1
\end{equation}
where $k= C_7 \left\|f\right\| _{RBMO(\mu)}$.
\end{lem}

\begin{proof}
Let $f\in\RBMO(\mu)$ with $\left\|f\right\|_{\RBMO(\mu)}\neq 0$. We have
\begin{equation*}
\int_{\mathbb R^{d}}\frac{e^{\frac{\left|f(x)-f_{\tilde{\mathbb Q}}\right|}{C_{6}\left\|f\right\|_{\RBMO(\mu)}}}-1}{\left(1+\left|x\right|\right)^{2n+\kappa}}d\mu(x)=\int_{\mathbb Q}\frac{e^{\frac{\left|f(x)-f_{\tilde{\mathbb Q}}\right|}{C_{6}\left\|f\right\|_{\RBMO(\mu)}}}-1}{\left(1+\left|x\right|\right)^{2n+\kappa}}d\mu(x)+\int_{\mathbb Q^{c}}\frac{e^{\frac{\left|f(x)-f_{\tilde{\mathbb Q}}\right|}{C_{6}\left\|f\right\|_{\RBMO(\mu)}}}-1}{\left(1+\left|x\right|\right)^{2n+\kappa}}d\mu(x),
\end{equation*}
where $\mathbb Q^{c}=\mathbb R^{d}\setminus \mathbb Q$. The first term in the right hand side is less that $\mu(\rho\mathbb Q)$. For the second term, we have
\begin{eqnarray*}
\int_{\mathbb Q^{c}}\frac{e^{\frac{\left|f(x)-f_{\tilde{\mathbb Q}}\right|}{C_{6}\left\|f\right\|_{\RBMO(\mu)}}}-1}{\left(1+\left|x\right|\right)^{2n+\kappa}}d\mu(x)&=&\sum^{\infty}_{k=0}\int_{2^{k+1}\mathbb Q\setminus2^{k}\mathbb Q}\frac{e^{\frac{\left|f(x)-f_{\tilde{\mathbb Q}}\right|}{C_{6}\left\| f\right\|_{\RBMO(\mu)}}}-1}{\left(1+\left|x\right|\right)^{2n+\kappa}}d\mu(x)\\
&\leq&C\sum^{\infty}_{k=0}2^{-(2n+\kappa) k}\int_{2^{k+1}\mathbb Q}\left(e^{\frac{\left|f(x)-f_{\tilde{\mathbb Q}}\right|}{C_{6}\left\|f\right\|_{\RBMO(\mu)}}}-1\right)d\mu(x).
\end{eqnarray*} 
 Furthermore, there exists a non-negative constant $K$ such that 
 \begin{equation}
\left|f_{\tilde{R}}-f_{\tilde{Q}}\right|\leq KS_{Q,R}\left\|f\right\|_{RBMO(\mu)}\text{ for two cubes }Q\subset R,\label{relationQR}
\end{equation}
as we can see in the proof of Lemma 2.8 in \cite {T1}. We also have $S_{\mathbb Q,2^{k+1}\mathbb Q}\leq (k+2)$, which leads to $\left|f_{\tilde{\mathbb Q}}-f_{\tilde{2^{k+1}\mathbb Q}}\right|\leq \log(2^{\frac{K}{\log 2}(k+2)})\left\|f\right\|_{\RBMO(\mu)}$.

Hence
\begin{equation*}
\begin{aligned}
&\sum^{\infty}_{k=0}2^{-(2n+\kappa) k}\int_{2^{k+1}\mathbb Q}\left(e^{\frac{\left|f(x)-f_{\tilde{\mathbb Q}}\right|}{C_{6}\left\|f\right\|_{\RBMO(\mu)}}}-1\right)d\mu(x)\\
&\ \ \ \ \ \ \ \ \ \ \ \ \ \ \leq\sum^{\infty}_{k=0}2^{-(2n+\kappa) k}2^{\frac{K}{C_{6}\log 2}(k+2)}\int_{2^{k+1}\mathbb Q}\left(e^{\frac{\left|f(x)-f_{\tilde{2^{k+1}\mathbb Q}}\right|}{C_{6}\left\|f\right\|_{\RBMO(\mu)}}}-1\right)d\mu(x)\\
&\ \ \ \ \ \ \ \ \ \ \ \ \ \ \leq C\sum^{\infty}_{k=0}2^{(-n-\kappa+\frac{K}{C_{6}\log 2})k}.
\end{aligned}
\end{equation*}
If we choose $C_{6}>\frac{K}{(n+\kappa)\log 2}$ then the above series converges. Finally we have 
\begin{equation}
\int_{\mathbb R^{d}}\frac{e^{\frac{\left|f(x)-f_{\tilde{\mathbb Q}}\right|}{C_{6}\left\|f\right\|_{\RBMO(\mu)}}}-1}{\left(1+\left|x\right|\right)^{2n+\kappa}}d\mu(x)\leq K_{1},
\end{equation}
where $K_{1}$ is a non-negative constant not depending on $f$.

 Thus the result follows from taking $C_{7}=\max(C_{6},K_{1}C_{6})$.
\end{proof}

\section{Some properties of Orlicz and Hardy-Orlicz space}\label{hardy-orlicz}
For the definition of Hardy-Orlicz space, we need the maximal characterization of $\H^{1}(\mu)$ given in \cite{T2}. 

Let $f\in L^{1}_{loc}(\mu)$, we set
 \begin{equation}
 \M f(x)=\sup_{\varphi\in F(x)}\left|\int_{\bR^{d}}f\varphi d\mu\right|,
 \end{equation}
where for $x\in\bR^{d}$, $F(x)$ is the set of $\varphi\in L^{1}(\mu)\cap\mathcal C^{1}(\bR^{d})$ satisfying the following conditions: 
\begin{equation}
	 \left\|\varphi\right\|_{L^{1}(\mu)}\leq 1,\label{eqa1}
	\end{equation}
	\begin{equation}
	0\leq\varphi(y)\leq\frac{1}{\left|y-x\right|^{n}}\text{ for all }y\in\bR^{d}\label{eqa2}
	\end{equation}
	and
	\begin{equation}
	\left|\nabla\varphi(y)\right|\leq\frac{1}{\left|y-x\right|^{n+1}}\text{ for all }y\in\bR^{d}.\label{eqa3}
	\end{equation}
	Tolsa proved in Theorem 1.2 of \cite{T2} that a function  $f\in L^{1}(\mu)$ belongs to the Hardy space $\H^{1}(\mu)$ if and only if $\int_{\bR^{d}}fd\mu=0$ and $\M f\in L^{1}(\mu)$. Moreover, in this case we have
 \begin{equation} \left\|f\right\|_{\H^{1}(\mu)}\approx\left\|f\right\|_{L^{1}(\mu)}+\left\|\M f\right\|_{L^{1}(\mu)}.\label{norm1}
 \end{equation}
 
 Hardy-Orlicz spaces are defined via this maximal characterization. We recall that for a continuous function ${\mathcal P} :[0,\infty)\rightarrow [0, \infty)$ increasing from zero to infinity (but not necessarily convex), the Orlicz space $L^{\mathcal P} (\mu)$ consists of $\mu$-measurable functions
$f:\Omega\rightarrow {\mathbb R}$ such that
\begin{equation}
\left\|f\right\|_{L^{\mathcal P} (\mu)}:=\inf \left\{k >0: \int_{\bR^{d}} {\mathcal P}
 \left( k^{-1} \left|f\right|\right)d \mu \leq 1\right\}
<\infty.
\end{equation}
 In general, the nonlinear functional $\left\|\cdot\right\| _ {L^{\mathcal P}(\mu)}$ need not satisfy the triangle inequality. It is well known that $L^ {\mathcal P} (\mu)$ is a complete linear metric space, see \cite{RR}. The $ L^\mathcal P$-distance between $f$ and $g$ is given by 
 \begin{equation}\label{eq1.16bis}
\textrm{dist}_\mathcal P[f,g] :=\inf \left\{\rho >0:
\int_{\bR^{d}} {\mathcal P}
 \left(\rho^{-1}|f-g|\right)d \mu\leq\rho\right\}
<\infty\,.
\end{equation}

\smallskip

The Hardy-Orlicz space $\H^{\mathcal P}(\mu)$ consists of local integrable function  $f$ such that $\M f\in L^ {\mathcal P} (\mu)$. We
put
\begin{equation}
\left\|f\right\|_{\H^ {\mathcal P}(\mu)}=\left\|\M f\right\|_{L^{\mathcal P} (\mu)}.
\end{equation}
It comes from what precede that $\H^\mathcal P(\mu)$ is  a complete linear metric space, a Banach space when ${\mathcal P}$ is convex. These spaces
have previously been dealt with by many authors, see \cite{BoM,Ja2, St} and further references given there.
When we consider the Orlicz function $\wp(t)=\dfrac{t}{\log(e+t)}$, we have the following results given in \cite{BIJZ}.

$\bullet$ If $\textrm{dist}_{\wp}[f,g]\leq1$ then $\left\| f-g\right\|_{L^{\wp}(\mu)}\leq \textrm{dist}_{\wp}[f,g]\leq1$, 

$\bullet$ The sequence $(f_{j})_{j>0}$ converges to $f$ in $L^\wp(\mu)$ if and only if $\left\| f_j-f\right\|_{L^{\wp}(\mu)} \rightarrow 0$,

$\bullet$ We have duality between the Orlicz space $L^{\Xi} (\mu)$ associated to the Orlicz function $ {\Xi}(t)=e^t-1$ and $L^{\tilde{\wp}}(\mu)$ with $\tilde{\wp}(x)=x\log(e+x)$ in the sense that
for $f\in L^{\Xi} (\mu)$ and $g\in L^{\tilde{\wp}}(\mu)$ we have 
\begin{equation}
\left\|fg\right\|_{L^{1}(\mu)}\leq\left\|f\right\|_{L^{\Xi} (\mu)}\left\|g\right\|_{L^{\tilde{\wp}}(\mu)}.
\end{equation}
$\bullet$ For $f,g\in L^{\wp}(\mu)$,we have the following substitute of the additivity
\begin{equation}
\left\|f+g\right\|_{L^{\wp}(\mu)}\leq 4\left\|f\right\|_{L^{\wp}(\mu)}+4\left\|g\right\|_{L^{\wp}(\mu)}.\label{additivity}
\end{equation}
$\bullet$ Let 
\begin{equation}
d\sigma=\frac{d\mu}{(1+|x|)^{2n+\kappa}}\text{  and  } d\nu=\frac{d\mu}{\log(e+\left|x\right|)},
\end{equation}
for  $f\in L^{\Xi} (\sigma)$ and $g\in L^{1}(\mu)$, we have $fg\in L^\wp(\nu)$ and 
\begin{equation}
\left\|fg\right\|_{L^{\wp}(\nu)}\leq C\left\|f\right\|_{L^{\Xi} (\sigma)}\left\|g\right\|_{L^{1}(\mu)}.
\end{equation}
and for  $f\in \RBMO(\mu)$ and $g\in L^{1}(\mu)$,
\begin{equation}
\left\|fg\right\|_{L^{\wp}(\nu)}\leq C\left\|f\right\|_{\RBMO(\mu)^{+}}\left\|g\right\|_{L^{1}(\mu)},\label{holder3}
\end{equation}
where $\left\|f\right\|_{\RBMO^{+}(\mu)}=\left\|f\right\|_{\RBMO(\mu)}+\left|f_{\tilde{\mathbb Q}}\right|$
\section{Proof of the main results}
\begin{proof}[{\bf Proof of Theorem \ref{main1}}]
Let $f\in\RBMO(\mu)$ and $h\in\H^{1}(\mu)$, $h$ having the $p$-atomic blocks decomposition given in (\ref{atomicblockdecompo}), i.e. 
\begin{equation}
h=\sum_{j} b_{j},
\end{equation}
where $b_{j}=\sum^{\infty}_{i=1}\lambda_{ij}a_{ij}$ is the atomic-block supported in the cube $R_{j}$, $a_{ij}$ supported in the cube $Q_{ij}\subset R_{j}$ and $\left\|a_{ij}\right\|_{L^{\infty}(\mu)}\leq \mu\left(\rho Q_{ij}\right)^{-1}\left(S_{Q_{ij},R_{j}}\right)^{-1}$. 

We have 
\begin{eqnarray*}
\left\|\lambda_{ij}\left(f-f_{\tilde{R}_{j}}\right)a_{ij}\right\|_{L^{1}(\mu)}&\leq&\left|\lambda_{ij}\right|\int_{Q_{ij}}\left|f-f_{\tilde{R}_{j}}\right|\left|a_{ij}\right|d\mu\\
&\leq& \left|\lambda_{ij}\right|\left(\int_{Q_{ij}}\left|f-f_{\tilde{Q}_{ij}}\right|\left|a_{ij}\right|d\mu+\int_{Q_{ij}}\left|f_{\tilde{R_{j}}}-f_{\tilde{Q}_{ij}}\right|\left|a_{ij}\right|d\mu\right)\\
&\leq& C\left|\lambda_{ij}\right|\left\|f\right\|_{RBMO(\mu)},
\end{eqnarray*}
according to Inequalities (\ref{relationQR}) and (\ref{a3}), which proves that the first series $\sum^{\infty}_{j=1}\left(f-f_{\tilde{R}_{j}}\right)b_{j}=\sum^{\infty}_{j=1}\sum^{\infty}_{i=1}\lambda_{ij}\left(f-f_{\tilde{R}_{j}}\right)a_{ij}$ converges normally in $L^{1}(\mu)$, since the atomic decomposition theorem asserts that the double series $\sum_{i,j}\left|\lambda_{ij}\right|$ converges. It remains to prove the convergence of
\begin{equation} 
S=\sum^{\infty}_{j=1}\left(\sum^{\infty}_{i=1}\lambda_{ij}f_{\tilde{R}_{j}}a_{ij}\right)=\sum^{\infty}_{j=1}f_{\tilde{R}_{j}}b_{j}
\end{equation}
in $\H^{\wp}(\nu)$. For this purpose, we have to prove that the sequence $S_{N}=\M \left(\sum^{N}_{j=1}f_{\tilde{R}_{j}}b_{j}\right)$ is Cauchy in $L^{\wp}(\nu)$.  This is equivalent to prove that $\underset{{l\rightarrow\infty}}{\lim}\left\|\M \left(\tilde{S}^{k}_{l}\right)\right\|_{L^{\wp}(\nu)}=0$, where 
\begin{equation}
\tilde{S}^{k}_{l}=\sum^{k}_{j=l}f_{\tilde{R}_{j}}b_{j}\text{ with } l\leq k.
\end{equation}
Since 
\begin{equation}
\M \left(f_{\tilde{R}_{j}}b_{j}\right)\leq\left|f-f_{\tilde{R}_{j}}\right|\M (b_{j})+\left|f\right|\M (b_{j}),
\end{equation}
we have that 
\begin{equation}
\left\|\M \left(\tilde{S}^{k}_{l}\right)\right\|_{L^{\wp}(\nu)}\leq4\left\|\sum^{k}_{j=l}\left|f-f_{\tilde{R}_{j}}\right|\M (b_{j})\right\|_{L^{1}(\mu)}+4\left\|\sum^{k}_{j=l}\left|f\right|\M (b_{j})\right\|_{L^{\wp}(\nu)},\label{I2}
\end{equation}
according to (\ref{additivity}) and the fact that $\left\|f\right\|_{L^{\wp}(\mu)}\leq\left\|f\right\|_{L^{1}(\mu)}$ for all measurable functions $f$. Let us consider the first term in the second member of (\ref{I2}). We have 
\begin{equation}
\left\|\sum^{k}_{j=l}\left|f-f_{\tilde{R}_{j}}\right|\M (b_{j})\right\|_{L^{1}(\mu)}\leq\sum^{k}_{j=l}\left\|\sum^{\infty}_{i=1}\left|\lambda_{ij}\right|\left(\left|f-f_{\tilde{Q}_{ij}}\right|+\left|f_{\tilde{R}_{j}}-f_{\tilde{Q}_{ij}}\right|\right)\M (a_{ij})\right\|_{L^{1}(\mu)},
\end{equation}
since $\M (b_{j})\leq\sum^{\infty}_{i=1}\left|\lambda_{ij}\right|\M(a_{ij})$. From the definition of $\M(a_{ij})$, we have 
\begin{equation}
 \M(a_{ij})(x)\leq\mu\left(\rho Q_{ij}\right)^{-1}\left(S_{Q_{ij},R_{j}}\right)^{-1},
\end{equation}
so that taking into consideration relation (\ref{relationQR}), we obtain
\begin{equation}
\left\|\left(\left|f-f_{\tilde{Q}_{ij}}\right|+\left|f_{\tilde{R}_{j}}-f_{\tilde{Q}_{ij}}\right|\right)\M (a_{ij})\right\|_{L^{1}(\mu)}\leq C\left\|f\right\|_{RBMO(\mu)}.
\end{equation}
Thus 
\begin{equation}
\lim_{l\rightarrow\infty}\left\|\sum^{k}_{j=l}\left|f-f_{\tilde{R}_{j}}\right|\M (b_{j})\right\|_{L^{1}(\mu)}=0,
\end{equation}
since the double series $\sum^{\infty}_{j=1}\left(\sum^{\infty}_{i=1}\left|\lambda_{ij}\right|\right)$ converges. Let us consider now the series $$\left\|\sum^{k}_{j=l}\left|f\right|\M (b_{j})\right\|_{L^{\wp}(\nu)}=\left\|\left|f\right|\sum^{k}_{j=l}\M (b_{j})\right\|_{L^{\wp}(\nu)}.$$
We have
\begin{equation}
\left\|\sum^{k}_{j=l}\M (b_{j})\right\|_{L^{1}(\mu)}\leq C\sum^{k}_{j=l}\left(\sum^{\infty}_{i=1}\left|\lambda_{ij}\right|\right),
\end{equation}
according to Lemma 3.1 of \cite{T2}. Furthermore, we have
\begin{equation}
\left\|\left|f\right|\sum^{k}_{j=l}\M (b_{j})\right\|_{L^{\wp}(\nu)}\leq\left\|f\right\|_{\RBMO^{+}(\mu)}\left\|\sum^{k}_{j=l}\M (b_{j})\right\|_{L^{1}(\mu)},
\end{equation}
according to (\ref{holder3}).
\end{proof}

\begin{defn}(\cite{BF}) $L^{\wp}_{*}$ is the space of functions $f$ such that
$$\|f\|_{L^{\wp}_{*}}:=\sum_{j\in \bZ^n}\|f\|_{L^{\wp}(j+\mathbb Q)}<\infty,$$
where $\mathbb Q$ is the unit cube centered at $0$.
\end{defn}

 We accordingly define $\H^{\wp}_{*}$. Using the concavity described above, we have $\wp(st)\leq C s\wp(t)$ for $s>1$. It follows that $L^\wp$ is contained in $L^{\wp}_{*}$ as a consequence of the fact that $\|f\|_{L^{\wp} (j+\mathbb Q)}\leq \int_{j+\mathbb Q} \wp(|f|)d\mu(x)$. The converse inclusion is not true.

\begin{thm}\label{main2}
For $h\in\mathcal H^1(\mu)$ and $f\in\rbmo (\mu)$, the product $f\times h$ can be given a meaning in
the sense of distributions. Moreover, we have the inclusion
\begin{equation}
  f\times h\in L^1(\mu)+ \mathcal H^{\wp}_{\ast}(\mu).
\end{equation}
\end{thm}
\begin{proof}
The proof is inspired by the one given in \cite{BF} in the case of Lebesgue measure.
Let $f\in\rbmo(\mu)$ and $h\in\H^1(\mu)$ being as in the proof of Theorem \ref{main1}.
 The series 
 \begin{equation}
 \sum_{j}\left(\sum_{i}\lambda_{ij}(f-f_{\tilde{R}_{j}})a_{ij}\right),\ \sum_{j}(f-f_{\tilde{R}_{j}})\M (b_{j})\text{ and }\sum_{j}\M (b_{j})
 \end{equation}
 converge normally in $L^{1}(\mu)$ and
 \begin{equation}
 \M \left(\sum_{j}b_{j}f_{\tilde{R}_{j}}\right)\leq\sum_{j}\left|f-f_{\tilde{R}_{j}}\right|\M (b_{j})+\left|f\right|\sum_{j}\M (b_{j}).\label{decompo}
 \end{equation}
  Thus we just have to prove that the second term in the right hand side of (\ref{decompo}) is in $L^{\wp}_{\ast}(\mu)$. Let $Q$ be a cube of side length 1. By John-Nirenberg inequality on $\rbmo(\mu)$, we have that there exists $c_{7}>0$ (we can choose any number greater than $\frac{1}{c_{5}}+\frac{c_{4}2^{n}}{c_{5}}$) such that 
  \begin{equation}
  \int_{Q}\left(e^{\frac{\left|f(x)\right|}{c_{7}\left\|f(x)\right\|_{\rbmo(\mu)}}}-1\right)d\mu(x)\leq 1.
  \end{equation}
  We claim that for $\psi\in L^{1}(\mu)$
  \begin{equation}
  \left\|f\psi\right\|_{L^{\wp}(Q)}\leq C\left\|f\right\|_{\rbmo(\mu)}\int_{Q}\left|\psi\right|d\mu.
  \end{equation}
  In fact, by homogeneity, we can assume that $c_{7}\left\|f\right\|_{\rbmo(\mu)}=1$ and it is sufficient to find some constant $c$ such that for $\int_{Q}\left|\psi\right|d\mu=c$ we have   
 $$\int_Q \frac {|f\psi|}{\log (e+|f\psi|)}d\mu\leq 1.$$ 
 We have 
 \begin{equation} \int_{Q}\frac{\left|f\psi\right|}{\log(e+\left|f\psi\right|)}d\mu=\int_{Q\cap\left\{\left|f\right|\leq1\right\}}\frac{\left|f\psi\right|}{\log(e+\left|f\psi\right|)}d\mu+\int_{Q\cap\left\{\left|f\right|>1\right\}}\frac{\left|f\psi\right|}{\log(e+\left|f\psi\right|)}d\mu.
 \end{equation}
 The first term in the second member is bounded by $\int_{Q}\left|\psi\right|d\mu$ and for the second term, we have 
 \begin{eqnarray*} \int_{Q\cap\left\{\left|f\right|>1\right\}}\frac{\left|f\psi\right|}{\log(e+\left|f\psi\right|)}d\mu&\leq&\int_{Q\cap\left\{\left|f\right|>1\right\}}\left|f\right|\frac{\left|\psi\right|}{\log(e+\left|\psi\right|)}d\mu\\
 &\leq&\left\|f\right\|_{L^{\Xi}(Q)}\left\|\frac{\left|\psi\right|}{\log(e+\left|\psi\right|)}\right\|_{L^{\tilde{\wp}}(Q)}\leq C\left\|\frac{\left|\psi\right|}{\log(e+\left|\psi\right|)}\right\|_{L^{\tilde{\wp}}(Q)}.
 \end{eqnarray*}
 But
 \begin{eqnarray*}
 \int_{Q}\frac{\left|\psi\right|}{\log(e+\left|\psi\right|)}\log\left(e+\frac{\left|\psi\right|}{\log(e+\left|\psi\right|)}\right)d\mu&\leq& \int_{Q}\frac{\left|\psi\right|}{\log(e+\left|\psi\right|)}\log\left(e+\left|\psi\right|\right)d\mu\\
 &\leq& \int_{Q}\left|\psi\right|d\mu\\
 \end{eqnarray*}
 Thus if $c<\frac{1}{2}$ and $\int_{Q}\left|\psi\right|d\mu=c$ the result follows. We have an estimate for each cube $j+\mathbb Q$, and sum up.
This finishes the proof.
\end{proof}
Since we do not have any maximal function characterization of the local Hardy spaces on non-homogeneous space in the literature, we are going to define the local space corresponding to $\H^{1}_{\ast}$ in the same manner as in \cite{BF}. For this purpose, we put
\begin{equation}
\M^{(1)}f(x)=\sup_{F_{loc}(x)}\left|\int f\varphi d\mu\right|,
\end{equation}
where $F_{loc}(x)$ denotes the set of elements belonging to $F(x)$ as define in Section \ref{hardy-orlicz}, but having their support in the cube $Q(x,1)$ centered at $x$ with side length $1$.  A locally integrable function $f$ belongs to the space $\h^{\wp}_{*}(\mu)$ if $\M^{(1)}f\in L^{\wp}_{\ast}(\mu)$.

\begin{prop}
For $h$ a function in $\h^1(\mu)$ and $b$ a function in $\rbmo(\mu)$, the product $b\times h$ can be given a meaning in the
sense of distributions. Moreover, we have the inclusion
\begin{equation}\label{inclusion1}
  b\times h\in L^1(\mu)+ \h^{\wp}_{*}(\mu).
\end{equation}
\end{prop}
\begin{proof}
Let $f\in\rbmo(\mu)$ and $h\in\h^1(\mu)$ with $h=\sum_{j}b_{j}$ where $b_{j}$'s are atomic blocks or blocks. Since we do not use the cancellation property of $b_{j}$'s to prove that the $\sum_{j}(f-f_{\tilde{R}_{j}})b_{j}$ converge absolutely in $L^{1}(\mu)$, it follows that the result remains true in this case. Thus we just have to prove that the second term belongs to the amalgam space $\h^{\wp}_{*}(\mu)$. This immediate if we prove that for any bock $b_{j}$, the quantity $\left\|\M^{(1)}b_{j}\right\|_{L^{1}(\mu)}$ is bounded by a constant which is independent on $b_{j}$. Let $b_{j}=\sum^{\infty}_{i=1}\lambda_{ij}a_{ij}$, where $a_{ij}$ is supported in the cube $Q_{ij}\subset R_{j}$ and satisfy $\left\|a_{ij}\right\|_{L^{\infty}(\mu)}\leq \left(\mu(2Q_{ij})S_{Q_{ij},R_{j}}\right)^{-1}$. For every integer $i$, we have 
\begin{equation}
\M^{(1)}a_{ij}(x)\leq \left(\mu(2Q_{ij})S_{Q_{ij},R_{j}}\right)^{-1}\chi_{2R_{j}}(x),
\end{equation}
where $\chi_{2R_{j}}$ denote the characteristic function of $2R_{j}$. In fact, if $\varphi\in F_{loc}(x)$ then $\int a_{ij}\varphi d\mu\neq 0$ only if $x\in 2R_{j}$, since $\ell(R_{j})>1$. Proceeding as in the prove of Proposition 2.6 in \cite{T2}, we have
\begin{equation}
\int_{\bR^{d}}\M^{(1)}a_{ij}(x)d\mu(x)=\int_{2R_{j}}\M^{(1)}a_{ij}(x)d\mu(x)\leq C,
\end{equation}
where $C$ is independent of $i$ and $j$. Then we conclude as in the proof of Theorem \ref{main2}. 
\end{proof}
\noindent{\bf Acknowledgments.}  I would like to thank the referee for his through revision of the paper and his useful comment. I would also like to thank professors Aline Bonami and Xavier Tolsa for some helpful discussions.

\end{document}